\documentclass{amsart}
\usepackage[utf8]{inputenc}

\usepackage{amsmath,amsthm,amssymb,latexsym,graphicx}

\usepackage[mathscr]{euscript}

\usepackage{amsfonts,enumitem, txfonts}

\usepackage{float}

\usepackage{hyperref}
\hypersetup{
colorlinks=true,
    linkcolor=blue,
    urlcolor = blue}

\newtheoremstyle{noperiod}
{3pt}
{3pt}
{\em}
{}
{\bf}
{}
{.2em}
{}

\newtheorem{thm}{Theorem}[section]
\newtheorem{mainThm}{Theorem}

\newtheorem{lemma}{Lemma}[section]

\newtheorem{corlemma}{Corollary}[lemma]
\newtheorem*{cor*}{Corollary}
\newtheorem{defn}{Definition}[section]
\newtheorem*{defn*}{Definition}
\newtheorem{fact}{Fact}[section]
\newtheorem{claimlemma}{Claim}[lemma]
\newtheorem*{prop*}{Proposition}
\newtheorem*{thm*}{Theorem}
\newtheorem*{lemma*}{Lemma}
\newtheorem{rmk}{Remark}[section]
\newtheorem{ex}{Example}[section]
\newtheorem{prob}{Problem}[section]

\theoremstyle{noperiod}

\theoremstyle{noperiod}

\theoremstyle{noperiod}

\theoremstyle{noperiod}

\theoremstyle{noperiod}

\theoremstyle{definition}

\theoremstyle{definition}

\DeclareMathOperator{\Aut}{Aut}

\newcommand{\aclass}{\mathcal{A}^{\delta}_{K_1,K_2,C,C',\mathcal{S}}}

\newcommand{\leqP}{\leq^{+}}

\title{The Algebra of an Age for Metrically Homogeneous Graphs of Generic Type}

\author{Rebecca Coulson}

\begin{document}

\maketitle

\begin{abstract}
Metrically homogeneous graphs are connected graphs which, when endowed with the path metric, are homogeneous as metric spaces. Here we consider a class of countable metrically homogeneous graphs. The algebra of an age is a concept introduced by Cameron in \cite{Cam_Orbs2} and is closely connected to the profile of the automorphism group of the associated countable structure. Cameron in \cite{Cam_AA} provided sufficient structural conditions on the age of $\aleph_0$-categorical countable homogeneous structures for showing that the algebra of the age is a polynomial algebra. In this paper, we use Cameron's result to deduce that the algebra of the age of certain metrically homogeneous graphs of generic type are polynomial algebras, typically in infinitely many variables.
\end{abstract}



\section{Introduction}
A group is said to be \emph{oligomorphic} if the number of orbits of $G$ acting on $n$-element sets of $S$ is finite for every $n$. Oligomorphic groups are of interest to model theorists, as they are precisely the automorphism groups of $\aleph_0$-categorical countable structures in countable languages. There is a very rich theory concerning the so-called \emph{profile} of an oligomorphic permutation group. The profile is the function giving the number of orbits of the group on (unordered) sets of order $n$; thus, if a permutation group is oligomorphic, then the profile is finite everywhere. From a model theoretic perspective, the profile function $F^*_n(T)$ of an oligomorphic $G$, for which $G = \Aut(M)$ where $M$ is the unique countable model of $T$, counts the number of $n$-types over $T$, up to a permutation of the entries.

Considerable attention has been paid to the asymptotic behavior of the profile function, and a more precise description has been sought for the possibilities in the polynomially bounded case. A powerful algebraic tool was introduced by Peter Cameron in \cite{Cam_Orbs2}, a graded $\mathbb{Q}$-algebra $\mathcal{A}^G$ (where $G$ is the permutation group in question) whose Hilbert function is the generating function for the profile. The definition of this algebra is given below. It may be described succinctly as the ring of $G$-invariants in the incidence algebra of the partially ordered set of finite subsets of the domain $\Gamma$ on which $G$ acts.

The following questions are asked about this algebra.
\begin{itemize}
    \item When is it an integral domain?
    \item When is it finitely generated?
    \item When is it a polynomial ring over $\mathbb{Q}$ (typically with infinitely many generators)?
\end{itemize}
Moreover, relationships between algebraic properties of $\mathcal{A}^G$, group theoretic properties of $G$, and combinatorial properties of the profile are examined. Already in \cite{Cam_Orbs2}, Cameron conjectured that the algebra $\mathcal{A}^G$ is an integral domain if and only if $G$ has no finite orbits. This was proved by Pouzet \cite{Pou_OrbIntDom}. In the case in which the growth rate of the profile is polynomially bounded, MacPherson asked in \cite{Mac_Growth} whether the associated algebra must be finitely generated, and in \cite{Cam_OligPer}, Cameron asked whether the profile must then be asymptotically polynomial. A positive solution to these questions, and considerably more, has been announced by Falque and Thiery \cite{FT}.

In the case that $G$ is the automorphism group of a homogeneous structure $\Gamma$ in a finite relational language, Cameron showed in \cite{Cam_AA} that relatively straightforward structural hypotheses on the age of $\Gamma$ will yield that $\mathcal{A}^G$ is a polynomial algebra (Theorem \ref{lemma:Cam:polyAlg}, below).

These structural properties in the context of metrically homogeneous graphs of generic type were explored by the author in \cite{C-TopDynRamThe}, and more generally by Aranda et. al in \cite{Czechs}, for reasons unrelated to the algebra of an age. We use these properties here to show that the algebra of the age of certain metrically homogeneous graphs of generic type are indeed polynomial algebras.

Specifically, we will prove the following.
\begin{mainThm}\label{thm:genPolyAlg}
Let $(\delta,K_1,K_2,C,C',\mathcal{S})$ be an admissible parameter sequence with $K_1$ and $\delta$ finite, and let $\Gamma$ be the corresponding metrically homogeneous graph, with automorphism group $G$. If $C = 2\delta +1$, suppose that $\delta$ is even. Then the associated algebra $\mathcal{A}^G$ is a polynomial algebra in infinitely many variables.
\end{mainThm}

This result depends on two ingredients: a criterion introduced by Cameron in \cite{Cam_AA}, and a ``disjoint sum" operation for metrically homogeneous graphs suggested by the ``magic parameter" used in \cite{Czechs}. This magic parameter is used by Aranda et al. in order to complete compatible edge-labeled graphs to metrically homogeneous graphs in $\aclass$. Their approach is a more general version of the completion process we develop in \cite{C-TopDynRamThe}, which was developed for an altogether different purpose than that of Cameron in \cite{Cam_AA}.

\section{Background}\label{sec:algAge}

\subsection{Algebra of an age}

The algebra in question may be defined as follows.

\begin{defn}\cite{Cam_AA}\label{defn:algAge}
Let $\Omega$ be a set, and $G$ a group acting on $\Omega$. The \emph{reduced incidence algebra} $\mathcal{A}$ associated with the partial order of finite subsets of $\Omega$ \cite{Rota} is the graded $\mathbb{Q}$-algebra defined as follows. For each $n$ let $V_n$ be the vector space of $\mathbb{Q}$-valued functions on the set of $n$-element subsets of $\Omega$. Then $\mathcal{A} = \bigoplus V_n$ with multiplication determined by
$$(fg)(X) = \sum_{X = X_1 \sqcup X_2} f(X_1)g(X_2).$$ Then $G$ acts naturally on $\mathcal{A}$ and $\mathcal{A}^G$ denotes the subalgebra of $G$-invariant functions.
\end{defn}

The algebra $\mathcal{A}^G$ is sometimes referred to as the \emph{orbit algebra} of $G$.

\begin{rmk}
Equivalently, $\mathcal{A} = \bigoplus V_n^G$ where $V_n^G$ is the space of functions constant on $G$-orbits, which may be identified with the space of functions on the $G$-orbits. If $G$ is the automorphism group of a homogeneous structure $\Gamma$, then the orbits on sets of order $n$ are the isomorphism types of substructures of $\Gamma$ of order $n$, which make up the so-called ``age'' of $\Gamma$.
\end{rmk}

Cameron gave a criterion in purely structural terms sufficient to establish that the algebra is polynomial, with many applications, among them the case in which the structure $\Gamma$ is the random graph. We will show that his method applies also in our general case. 

We present a modified version of Cameron's framework, with slightly narrower assumptions than his.

\begin{defn}\label{defn:CamCri}
Let $\mathscr{C}$ be a class of finite structures, closed under isomorphism. We write $A \subseteq B$ for the partial substructure relation (e.g., subgraph is partial substructure, whereas induced subgraph is a substructure).
\begin{enumerate}
\item A \emph{decomposition operator} for $\mathscr{C}$ consists of a binary operation $+$ on $\mathscr{C}$ satisfying the following conditions.
\begin{itemize}
    \item \emph{Functorality}: On isomorphisms of structures in $\mathscr{C},$ we have that for any pair of isomorphisms $i \colon A \rightarrow A'$ and $j \colon B \rightarrow B',$ the operator $+$ satisfies
    $$A+ B \simeq A' + B';$$
    \item \emph{Additivity}: $|A + B| = |A| + |B|$;
    \item \emph{Unique decomposition}: The commutative semigroup $(\mathscr{C},+)$ is freely generated by its indecomposable elements.
\end{itemize}
\item A decomposition operator for $\mathscr{C}$ is \emph{free} if there is a partial order on $\mathscr{C}$ satisfying the following.
\begin{center}
    For $A$ in $\mathscr{C}$, if $A$ is partitioned into induced substructures $B_1,\cdots,B_k$, then $B_1 + \cdots + B_k \leq A$.
\end{center}
\end{enumerate}
\end{defn}

\begin{rmk}
If a decomposition operator on $\mathscr{C}$ is free then there is a canonical partial order $\leq$ associated with the theory. Namely, one considers the transitive closure of the relation $$B \leq^{+} A$$ defined on $\mathscr{C}$ by
\begin{center}
    $B = B_1 + \cdots + B_k$ for some partition of $A$ into induced substructures.
\end{center}
In particular, this partial order is also invariant under isomorphism. In practice however, one proves freeness by specifying a suitable partial order.
\end{rmk}

\begin{ex} The decomposition of graphs as disjoint sums of connected graphs is a decomposition theory. It is free with respect to the subgraph relation.
\end{ex}

More subtle examples are found in \cite{Cam_AA}.

The point of this is the following.

\begin{thm}\cite[Theorem 2.1]{Cam_AA}\label{lemma:Cam:polyAlg}
If $G = \Aut(\Gamma)$ is the automorphism group of a homogeneous structure for a finite relational language, and the age $\mathscr{C}$ of $\Gamma$ has a free decomposition operator, then the algebra $\mathcal{A}^G$ is the polynomial algebra with generators corresponding to the isomorphism types of indecomposable elements of $\mathscr{C}$.
\end{thm}

The statement given in \cite{Cam_AA} is phrased in more general terms.

\begin{ex}\cite[Example 1]{Cam_AA} The algebra associated with the random graph is a polynomial algebra.
\end{ex}

The following is immediate.

\begin{lemma}
Let $\mathscr{C}$ have the free decomposition operator $+$ and let $\mathscr{C}'$ be a $+$-closed hereditary subset of $\mathscr{C}$. Then $+$ is a free decomposition operator for $\mathscr{C}'$.
\end{lemma}

\begin{ex}\cite[Example 1 (cont.)]{Cam_AA}
The algebra associated with the generic $K_n$-free graph is a polynomial algebra.
\end{ex}

\subsection{Metrically homogeneous graphs}

A connected graph $\Gamma$ is said to be \textit{metrically homogeneous} if, when viewed as a metric space using the path metric, every finite partial isometry from $\Gamma$ to itself can be extended to a full isometry.

The finite metrically homogeneous graphs were completely classified by Cameron \cite{Cam_TRANS}. Cherlin gave a catalog  of the known examples of infinite metrically homogeneous graphs in \cite{Che-2P}, with some evidence for its completeness. We deal only with the case of metrically homogeneous graphs of diameter at least $3$, as the case of smaller diameter falls under the case of homogeneous graphs, which have already been classified (\cite{Gard,Shee,LaW-HG}).

A large class of metrically homogeneous graphs in Cherlin's catalog are those of \emph{generic type}.

\begin{defn}\label{defn:gen}
A metrically homogeneous graph $\Gamma$ is of \emph{generic type} if it satisfies the following two conditions.
\begin{itemize}
    \item The graph induced on the set of neighbors of a vertex is primitive.
    \item The graph induced on the set of common neighbors of a pair of vertices at distance $2$ contains an infinite independent set.
\end{itemize}
\end{defn}

The metrically homogeneous graphs of generic type are Fra\"{i}ss\'{e} limits, whose ages consist of finite integer-valued metric spaces. 

There is a uniform description of all known metrically homogeneous graphs of generic type, in terms of the class of \emph{$3$-constrained} metrically homogeneous graphs. A homogeneous structure is $3$-constrained if the associated amalgamation class is $3$-constrained in the sense that it is determined by a set of forbidden structures (constraints) of order at most $3$.

Thus the metrically homogeneous graphs we consider will determined by the triangles they forbid, together with certain technical constraints called ``Henson constraints," which consist of forbidden $(1,\delta)$-spaces, where $\delta$ the diameter of the metrically homogeneous graph. The forbidden triangles are determined by five parameters $(\delta,K_1,K_2,C,C')$. The Henson constraints are written $\mathcal{S}$; a given metrically homogeneous graph of generic type will be written as $\Gamma^{\delta}_{K_1,K_2,C,C',\mathcal{S}}$.

Cherlin established the following.
\begin{fact}\cite[Theorems 12.1 and 13.1]{CheCat}\label{fact:3constr}
Let $\Gamma$ be a $3$-constrained metrically homogeneous graph of generic type and diameter at least $3$. Then
\begin{enumerate}
\item The set of triangles embedding in $\Gamma$ is determined by the numerical parameters $\delta,K_1,K_2,C_0$, and $C_1$.
\item These parameters satisfy one of the three sets of numerical conditions given in Table \ref{Table:admissible}.
\end{enumerate}
Conversely, every such parameter sequence is realized by some $3$-constrained metrically homogeneous graph.
\end{fact}

\begin{table}[H]
\begin{description}
\item [Case (a)] (bipartite case). $K_1 = \infty$:
\begin{itemize}
    \item $K_2 = 0$ and $C_1 = 2\delta +1$.
\end{itemize}
\item[Case (b)] (low $C$). $K_1$ finite, $C \leq 2\delta + K_1$.
\begin{itemize}
    \item $C = 2K_1 + 2K_2 + 1 \geq 2\delta + 1$;
    \item $K_1 + 2K_2 \leq 2\delta -1$;
    \item If $C' > C +1$ then $K_1 = K_2$ and $3K_2 = 2\delta -1$.
\end{itemize}
\item[Case (c)] (high $C$). $C > 2\delta + K_1$.
\begin{itemize}
    \item $K_1 + 2K_2 \geq 2\delta -1$ and $3K_2 \geq 2\delta$;
    \item If $K_1 + 2K_2 = 2\delta -1$ then $C \geq 2\delta + K_1 + 2$;
    \item If $C' > C +1$, then $C \geq 2\delta + K_2$.
\end{itemize}
\end{description}

\caption{Admissible parameter choices with $\delta \geq 3$}
\label{Table:admissible}
\end{table}
We will have occasion to refer to the precise conditions on the numerical parameters shown in Table \ref{Table:admissible}. We call such sequences of numerical parameters \emph{admissible}.

The parameters $(\delta,K_1,K_2,C,C')$ determine the forbidden triangles as follows.

\begin{defn}\label{defn:Ks}
Let $\delta,K_1,K_2,C_0,C_1$ be numerical parameters, some of which may be infinite. The associated set $\mathcal{T}(\delta,K_1,K_2,C_0,C_1)$ of forbidden triangle types consists of all triangle types $(i,j,k)$ which violate one or more of the following constraints.
\begin{itemize}
\item $i,j,k \leq \delta$;
\item If the perimeter $p = i+j+k$ is odd, then $2K_1 < p < 2K_2 + 2\min(i,j,k)$.
\item If the perimeter $p \equiv \epsilon \pmod 2$, ($\epsilon = 0$ or $1$), then $p < C_{\epsilon}$.
\end{itemize}
\end{defn}

There are two ways in which metrically homogeneous graphs of generic type can be \emph{imprimitive}, that is, they carry a non-trivial equivalence relation invariant under the action of the automorphism group. Imprimitive metrically homogeneous graphs of generic type are bipartite or antipodal (or both). Antipodal metrically homogeneous graphs are ones in which every vertex $v$ has a unique vertex $v'$ at distance $\delta$. The imprimitive graphs are somewhat exceptional cases of metrically homogeneous graphs of generic type.

This all amounts to the following.
\begin{fact}\cite[Theorems 9 and 14]{Che-2P}
Let $(\delta,K_1,K_2,C,C',\mathcal{S})$ be an admissible sequence of parameters; in particular, $\mathcal{S}$ is a set of Henson constraints of the appropriate type. Then $\aclass$ is an amalgamation class, and thus one may speak of the associated metrically homogeneous graph $\Gamma^{\delta}_{K_1,K_2,C,C',\mathcal{S}}$.
\end{fact}

For much more on metrically homogeneous graphs, see \cite{CheCat}.

\section{Applying Cameron's Criterion}

\subsection{Direct sum operations} A straightforward generalization of the decomposition theories considered above for graphs is the following.

\begin{defn}\label{defn:plusE}
Let $\mathcal{L}$ be a relational language and $E$ a distinguished binary relation symbol in $\mathcal{L}$. For $\mathcal{L}$-structures $A,B$, define $$A +_E B$$ to be the structure consisting of the disjoint union of $A$ and $B$ together with all relations $E(a,b)$ and $E(b,a)$, for $(a,b)$ in $A \times B$. 
\end{defn}

\begin{lemma}\label{lemma:plusEisDecompOp}
With the notation of Definition \ref{defn:plusE}, the operation $+_E$ is a free decomposition operator on the class of finite $\mathcal{L}$-structures.
\end{lemma}

\begin{proof}
The operation is clearly functorial and additive.

If we associate to each structure $A \in \mathscr{C}$ the graph $A^c$ which is the graph complement of $A$ with edge relation $E$, then a decomposition of $A$ corresponds to a decomposition of $A^c$ as a disjoint sum. So unique decomposition follows.

Similarly, if $A$ is partitioned into induced subgraphs $B_1, \cdots, B_k$, then the disjoint sum $B_1^c, \cdots, B_k^c$ is contained in $A^c$. So we define $B \leqP A$ on $\mathscr{C}$ by $B^c \subseteq A^c$, and freeness follows.
\end{proof}

\begin{corlemma}
Let $\mathscr{C}$ be a hereditary class of binary relational structures. Let $E$ be a distinguished binary symmetric relation in the language. If $\mathscr{C}$ is closed under the operation $+_E$, then this operation provides a free decomposition theory for $\mathscr{C}$.
\end{corlemma}

In practice, the case one has in mind in the above is the following: $\mathscr{C}$ is the class of finite substructures of a homogeneous structure in a binary relational language $L$. The language $L$ consists of names for the orbits on pairs of distinct elements. Furthermore, $\mathscr{C}$ should have a transitive automorphism group (otherwise, the single relation $E$ would be replaced by a finite set of relations, complicating the notation).

In the case of graphs, $E$ is either the edge or the non-edge relation.

\subsection{The case of metrically homogeneous graphs}

The question now becomes, what sorts of generalized disjoint sum operations are available in the language of metrically homogeneous graphs. 

Here we replace the notation $+_E$ by the notation $+_i$, where $i$ is the distance corresponding to the binary relation $E$. Recall that in our context of metrically homogeneous graphs, the binary relations correspond to the distances in $[\delta]$.

The following is implicit in \cite{Czechs}.

\begin{lemma}
Let $\Gamma$ be a $3$-constrained metrically homogeneous graph with parameters $(\delta,K_1,K_2,C,C')$, where $\delta \geq 3$. Let $M \in [\delta]$. Then the following are equivalent.
\begin{itemize}
    \item The class $\mathscr{C}$ of finite substructures of $\Gamma$ is closed under the operation $+_M$.
    \item $\max(K_1,\delta/2) \leq M \leq \min(K_2,(C-\delta-1)/2)$.
\end{itemize}
\end{lemma}

\begin{proof} 
This result is covered by \cite[Observation 4.1]{Czechs}; we give here some additional details for the proof.

As any triangles occurring in a composition $A +_M B$ have type $(M,M,i)$ for some distance $i \leq \delta$, and all such triangles occur in some composition, the first item is equivalent to the requirement that all triangles of type $(M,M,i)$ embed into $\Gamma$.

The condition $$M \geq \delta/2$$ is necessary and sufficient to ensure that structures in $\mathscr{C} +_M \mathscr{C}$ satisfy the triangle inequality.

Similarly, the condition $2M + \delta < C$ is necessary and sufficient to ensure that $\mathscr{C} +_M \mathscr{C}$ respects the perimeter bound.

It remains to consider constraints on triangles of odd perimeter; we assume now that triangle types $(M,M,i)$ have $i$ odd. 

By considering triangles of type $(M,M,1)$, we find that the conditions $$K_1 \leq M \leq K_2$$ are necessary.

It remains to check their sufficiency. There are three conditions on triangles of type $(M,M,i)$ corresponding to the parameters $K_1, K_2$.
\begin{align*}
    2M + i \geq K_1 && M + i \leq 2K_2 + 2M && 2M \leq 2K_2 + 2i
\end{align*}
\begin{align*}
    2M+i \geq 2K_1 +1 && 2M + i \leq 2K_2 + 2M && 2M + i \leq 2K_2 + 2i
\end{align*}
If $K_1 \leq M$, the first inequality is satisfied. If $M \leq K_2$, then the additional condition $M \geq \delta/2$ yields the second inequality. The third inequality is immediate from the assumption $M \leq K_2$.

The lemma follows.
\end{proof}

We require a similar lemma for the general case, in which Henson constraints occur.

\begin{lemma}\label{lemma:plusMclosed}
Let $\Gamma$ be a metrically homogeneous graph of generic type and of known type, with associated parameters $(\delta,K_1,K_2,C,C',\mathcal{S})$, where $\delta \geq 3$. Let $M \in [\delta]$. Then the following are equivalent.
\begin{itemize}
    \item The class $\mathscr{C}$ of finite substructures of $\Gamma$ is closed under the operation $+_M$.
    \item $\max(K_1,\delta/2) \leq M \leq \min(K_2,(C-\delta-1)/2)$ and in addition
    \begin{itemize}
        \item If there is a constraint $H \in \mathcal{S}$ in which the distance $\delta$ occurs, then $M < \delta$.
    \end{itemize}
\end{itemize}
\end{lemma}

\begin{proof}
In general there are two notions of Henson constraint which apply: one in the case $C = 2\delta +1$, and one in the remaining cases.

When $C = 2\delta +1$, our conditions imply $M = \delta/2$ and in particular $\delta$ is even. In this setting, the Henson constraints involve distances $1$ and $\delta-1$.
Since $\delta$ is even, we have that $M \neq 1,\delta-1$ in this case. Thus the additional condition is both vacuous and unnecessary, and the previous lemma suffices.

So we come to the main case in which $C > 2\delta +1$ and $\mathcal{S}$ is a family of $(1,\delta)$-spaces. In this case, we have $M \geq \delta/2 > 1$, so if the distance $\delta$ does not occur in a Henson constraint, then once again the additional constraint is both vacuous and unnecessary.

We come down then to the case in which the distance $\delta$ does occur in some (minimal forbidden) Henson constraint $H$. If $M < \delta$, then no conflicts can arise. Conversely, such a constraint $H$ is itself $+_{\delta}$-decomposable, and as the factors are not forbidden, we require $M < \delta$.

This completes consideration of all cases.
\end{proof}

The following is closely related to \cite[Lemma 5.1]{Czechs}: the conditions for the completion process used there are slightly more restrictive than those required here, where we take only disjoint sums.

\begin{lemma}\label{lemma:plusMclosedAndParams}
The following conditions on an admissible sequence of parameters $(\delta,K_1,K_2,C,C',\mathcal{S})$ are equivalent.
\begin{itemize}
\item There is a parameter $M$ for which $\aclass$ is $+_M$-closed.
\item $\delta$ and $K_1$ are finite. If $C = 2\delta +1$, then $\delta$ is even.
\end{itemize}
\end{lemma}

\begin{proof}
We refer to the conditions on $M$ given in the previous lemma.

Suppose first a suitable parameter $M$ exists. The conditions $M \geq K_1,\delta/2$ imply that $K_1$ and $\delta$ are both finite. If $C = 2\delta +1$, the conditions $\delta/2 \leq M \leq (C-\delta-1)/2$ imply that $M = \delta/2$ and $\delta$ is even.

Conversely, with $\delta$ and $K_1$ finite, we use the minimum value $$M = \max(K_1,\lceil \delta/2 \rceil).$$ So we first require the numerical conditions $$\max(K_1,\lceil \delta/2 \rceil) \leq \min(K_2,(C-\delta-1)/2).$$ We know that $K_1 \leq K_2$ by definition, and $2\lceil \delta/2 \rceil \leq \delta + 1\leq C - \delta -1$ unless $C = 2\delta +1$, and in this case as $\delta$ is even, the required inequality still holds. The other two inequalities required are
\begin{align*}
    2 K_1 + \delta +1 &\leq C\\
    \lceil \delta/2 \rceil \leq K_2.
\end{align*}
Here one must examine the conditions on admissible parameters in detail. There are three cases of admissible parameters, the first of which has already been excluded. We give the conditions which separate these three cases along with some of the relevant side conditions which apply in each case.

\begin{enumerate}[label=\Roman*]
\item $K_1 = \infty$;
\item\label{item:II} $K_1 < \infty$, $C = 2K_1 + 2K_2 + 1 \leq 2\delta + K_1$; and $K_1 + K_2 \geq \delta$;
\item\label{item:III} $K_1 \leq \infty$, $C > 2\delta + K_1$; and $\delta \leq (3/2)K_2$.
\end{enumerate}

In case (\ref{item:II}), we have
\begin{align*}
    C = 2K_1 + 2K_2 + 1 &\geq 2K_1 + (K_1 + K_2) + 1 \geq 2K_1 + \delta + 1\\
    \delta &\leq K_1 + K_2 \leq 2K_2
\end{align*}
so the relevant inequalities hold in this case.

In case (\ref{item:III}), we have
\begin{align*}
    C \geq 2\delta + K_1 + 1 \geq \delta + 2K_1 + 1\\
    \delta/2 \leq (3/4) K_2 \leq K_2
\end{align*}
and again the relevant inequalities hold.

This disposes of the purely numerical constraints. The final point to check is the following: if $M = \delta$, then $\mathcal{S}$ does not contain a Henson constraint in which the distance $\delta$ occurs. By our choice of $M$, this would mean $$K_1 = \delta < \infty.$$
The characterization of admissibility in such cases implies that no Henson constraint in $\mathcal{S}$ involves the distance $\delta$. More precisely, in case (\ref{item:II}), $1$-cliques are allowed as Henson constraints, and in case (\ref{item:III}), no Henson constraints are allowed.

Thus all conditions are verified and the lemma follows.
\end{proof}

We may now apply the general theory to prove Theorem \ref{thm:genPolyAlg}.

\begin{proof}[Proof of Theorem \ref{thm:genPolyAlg}]
Our hypotheses on the parameters are those necessary for the application of Lemma \ref{lemma:plusMclosedAndParams}. So we have a value $M$ for which the associated class of finite structures is closed under $+_M$. Therefore by Lemma \ref{lemma:plusEisDecompOp}, the operator $+_M$ provides a free decomposition operator for the class.

By Cameron's criterion Theorem \ref{lemma:Cam:polyAlg}, the associated algebra is polynomial.

The indecomposable elements are those which are connected after deleting edges with weight $M$. For any $n$, any configuration consisting of a point $a$ and $n$ neighboring points is in $\mathcal{A}^{\delta}_{K_1,K_2,C,C',\mathcal{S}}$ and is connected with respect to weight $1$ edges. So there are infinitely many indecomposable isomorphism types and therefore the polynomial algebra has infinitely many generators.
\end{proof}

\subsection{The bipartite antipodal case}

We now examine one of the cases not covered by Theorem \ref{thm:genPolyAlg}.

\begin{lemma}\label{lemma:bipAnt}
Let $\Gamma$ be the generic bipartite antipodal graph of diameter $3$, with $G = \Aut(\Gamma)$. Then the associated algebra $\mathcal{A}^G$ is a polynomial algebra in three variables.

Furthermore, the associated class $\mathscr{C}$ has a free decomposition operator.
\end{lemma}

\begin{proof}
If $A$ is a finite bipartite antipodal graph with parts $A_1,A_2$, let $$\alpha(A) = (k,m,n)$$ where $m = \min(|A_1|,|A_2|)$, $n = \max(|A_1|,|A_2|)$, and $k$ is the number of antipodal pairs in $A$.

\begin{claimlemma}
The function $\alpha$ induces a bijection between the set of isomorphism types of bipartite finite antipodal graphs and the set $S$ of triples $(k,m,n)$ satisfying $k \leq m \leq n$.
\end{claimlemma}

We define a map $\beta$ from $S$ to bipartite finite antipodal graphs by setting $\beta(k,m,n) = (V_1,V_2)$ with $|V_1| = m$, $|V_2| = n$, $d(a_i,b_i) = 3$ for $a_i,b_i$ which are $k$ elements of $V_1,V_2$, respectively, and remaining distances $1$ between $V_1$ and $V_2$ and $2$ within $V_1$ or $V_2$.

Then $\alpha \circ \beta$ is the identity on $S$. We claim that $\beta \circ \alpha$ is also the identity.

If $\alpha(A) = (k,m,n)$, then we may suppose that $|A_1| = m$, $|A_2 = n$. Since there are exactly $k$ pairs $(a_i,b_i)$ at distance $3$, with $a_i \in A_1$ and $b_i \in B_2$, clearly $A \simeq B$. 

The claim follows.\\

Now $S$ is a semigroup under pointwise addition, and the elements $x = (0,0,1)$, $y = (0,1,1)$, $z = (1,1,1)$ are indecomposable. Any element $(k,m,n)$ may be written uniquely as 
$$k(1,1,1) + (m-k)(0,1,1) + (n-m)(0,0,1)$$ so the semigroup is freely generated by $x,y,z$.

We may transfer this semigroup structure to the age of $\Gamma$.

There is also a natural partial order $\leq$ on $S$ given by $(k_1,m_1,n_1) \leq (k_2,m_2,n_2)$ if $k_1 \leq k_2$, $m_1 \leq m_2$, and $m_1 + n_1 = m_2 + n_2$ (this last condition is unimportant but will hold in all cases of interest).

We transfer this partial order to the age of $\Gamma$ as well. Then the final assumption in Cameron's criterion (Definition \ref{defn:CamCri}) is that if $A$ is partitioned into induced substructures $B_1,\cdots,B_{\ell}$, with $\alpha(A) = (k,m,n)$ and $\alpha(B_i) = (k_i,m_i,n_i)$, then $$\left(\sum k_i, \sum m_i, \sum n_i\right) \leq (k,m,n),$$ i.e.,
\begin{align*}
    \sum k_i \leq k && \sum m_i \leq m && \sum m_i + \sum n_i = m + n.
\end{align*}
Clearly $\sum k_i \leq k$ by counting, and $m$ is the sum of terms $m_i$ with $m_i \leq n_i$, so the second inequality also holds. The final equality holds since the $B_i$ collectively partition $A$.

Thus Cameron's criterion applies, and the generators for $\mathcal{A}^G$ as a polynomial algebra correspond to the indecomposable elements $x,y,z$ of $S$.
\end{proof}

\subsection*{Open problems}
\begin{prob}
Is the associated algebra polynomial also in the remaining cases of the known metrically homogeneous graphs of generic type, namely the general case of bipartite graphs and of antipodal graphs of odd diameter?
\end{prob}

In such cases there is no operator of the form $+_M$ under which the class is closed, but as we have seen in Lemma \ref{lemma:bipAnt}, there may be a suitable free decomposition operator of another kind.

\begin{prob}
Let $\mathcal{F}$ be a finite set of finite connected graphs and suppose that there is an $\aleph_0$-categorical countable universal $\mathcal{F}$-free graph $\Gamma$. Is the associated algebra a polynomial algebra?
\end{prob}

To clarify, in this setting there is a \emph{canonical} $\aleph_0$-categorical countable universal $\mathcal{F}$-free graph (namely, the existentially complete one), and the question applies to this particular graph. Again Cameron's criterion suggests a natural approach to the problem.

\bibliographystyle{amsalpha}
\bibliography{References}\label{sec:biblio}

\end{document}